\documentclass[12pt,reqno,a4paper]{amsart}
\usepackage{amssymb}
\usepackage{amsthm}
\usepackage{tikz}
\usepackage{verbatim}
\usepackage{hyperref}
\usepackage[english]{babel}
\usepackage{wasysym}
\usepackage[shortlabels]{enumitem}
\usepackage[colorinlistoftodos, textsize=small]{todonotes}
\usepackage{color}
\usepackage{soul}
\usepackage{mathtools}
\usepackage{mathrsfs}
\usepackage{bm}
\usepackage{graphicx}
\usepackage{float}
\usetikzlibrary{arrows}

\newtheorem*{cla1}{Claim 1}
\newtheorem*{cla2}{Claim 2}
\newtheorem*{cla3}{Claim 3}
\newtheorem*{cla4}{Claim 4}

\newenvironment{claimproof}{\paragraph{\em Proof of claim.}}{\hfill\scalebox{0.7}{$\square$}}

\newtheorem{theorem}{Theorem}[section]

\newtheorem{lemma}[theorem]{Lemma}
\newtheorem{corollary}[theorem]{Corollary}
\newtheorem{proposition}[theorem]{Proposition}

\numberwithin{equation}{section}

\newcommand{\alg}[1]{\mathbf{#1}}
\newcommand{\rel}[1]{\mathbb{#1}}
\newcommand{\var}[1]{\mathcal{#1}}
\newcommand{\slt}{\mathcal{S}\mathcal{L}}
\newcommand{\dl}{\mathcal{D}\mathcal{L}}
\newcommand{\set}{\mathcal{S}\mathcal{E}\mathcal{T}}

\newcommand{\pol}{\operatorname{Pol}}
\newcommand{\clo}{\operatorname{Clo}}

\newcommand{\id}{\operatorname{id}}
\newcommand{\Id}{\operatorname{Id}}

\DeclareMathAlphabet\mathbfcal{OMS}{cmsy}{b}{n}

\begin{document}

\title[Connectivity notions on digraphs]
{Connectivity notions on compatible digraphs in equational classes}

\address{Bolyai Institute, Univ. of Szeged, Szeged, Aradi V\'{e}rtan\'{u}k tere 1, HUNGARY 6720}

\author{Gerg\H o Gyenizse} 
\email{gergogyenizse@gmail.com}    
           
\author{Mikl\'os Mar\'oti}
\email{mmaroti@math.u-szeged.hu} 
            
\author{L\'aszl\'o Z\'adori}
\email{zadori@math.u-szeged.hu}

\thanks{\thanks{The research of authors was supported by the NKFIH grants K138892 and  ADVANCED 153383,  and Project no TKP2021-NVA-09 where the latter has been financed by the Ministry of Culture and Innovation of Hungary from the National Research, Development and Innovation Fund.}}


\begin{abstract} 
A digraph $\rel G$ is called weakly connected, strongly connected, and extremely connected if any two vertices of $\rel G$ are connected respectively by an oriented, a directed, and a symmetric path in $\rel G$. We investigate the algebraic properties of digraphs that force some of these connectivity notions to coincide. 

We prove that for digraphs with a Hobby-McKenzie polymorphism, the strong and the extreme components coincide. Conversely, if the strong and the extreme components of any compatible digraph in an equational class of algebras coincide, then the class must have a Hobby-McKenzie term. As a consequence, we obtain that an equational class $\var V$ is $n$-permutable for some $n$ if and only if the weak components of any compatible reflexive digraph in $\var V$ are extremely connected.
\end{abstract}

\keywords{digraph; variety; interpretability; clone homomorphism; connectivity}

\maketitle

\section{Introduction}
In this article, we investigate the shape of compatible digraphs in various equational classes of algebras. Before we give motivation, we clarify some basic concepts that are used in our investigations. 

Throughout the text, we use bold face capitals and the same capitals to denote algebras and
their underlying sets, respectively, and  use blackboard bold capitals and the same capitals to denote digraphs and their vertex sets, respectively. We use calligraphic capital letters to denote varieties.

A digraph $\rel G$ is {\em weakly connected, strongly connected, extremely connected} if any two vertices of $\rel G$ are connected respectively by an oriented, a directed, and a symmetric path in $\rel G$. We need a fourth notion of connectedness, radically connected, that is a bit more complicated concept to define. We shall give the definition later in Section 2. All one needs to know for understanding the content of the Introduction is that for any digraph, the following implications hold: extremely connected $\Rightarrow$ radically connected $\Rightarrow$ strongly connected $\Rightarrow$ weakly connected. Each of these notions of connectivity naturally yields an equivalence relation on the vertex set of a digraph. We respectively call it the extreme, the radical, the strong, and the weak equivalence of the digraph. In Figure \ref{conn}, we displayed the inclusion ordering of these four equivalences of a digraph.

 \begin{figure}[H] 
\centering
\includegraphics[scale=1]{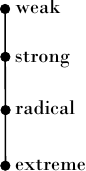}
\caption{The chain of equivalences related to the four connectivity notions.}
\label{conn}
\end{figure}

A {\em variety} (called also an {\em equational class}) is a class of all algebras of the same signature that satisfy a  set of identities in the given signature.  We call a digraph $\rel G$ {\em compatible} in a variety $\var V$ if there is an algebra $\alg G$ in $\var V$ whose underlying set coincides with the vertex set of $\rel G$,  and  the edge relation of $\rel G$ is a subalgebra of $\alg G^2$. Then, we also call $\rel G$  a {\em compatible digraph} of the algebra $\alg G$. 

Taylor varieties were introduced in \cite{GT} as the varieties whose full idempotent reduct does not interpret in the variety  $\set$ of sets. 
Taylor varieties were characterized in various ways, see for example, Theorem A.1. in \cite{KK}, and Lemma 4.2 and Corollary 4.3 in \cite{BGMZ0}. In  \cite{GMZ2}, we gave a characterization of Taylor varieties with the help of two connectivity notions of compatible reflexive digraphs. 
The equivalence of items (1) and (2) in the following theorem is a part of this characterization.  By taking into account, that radically connected implies strongly connected, and that the strong equivalence is the smallest equivalence which yields a cycle-free quotient of a digraph, items (2) and (3) are clearly equivalent. Item (3) suits more for our purposes in this article, as it gives direct connection between the equivalences related to two connectivity notions defined above. 

 \begin{theorem} [\cite{GMZ2}, cf. Theorem 4.4]
\label{th:taylor_equivalent_to_strong_equals_radical}
For any variety $\mathcal V$, the following are equivalent.
\begin{enumerate}
    \item $\var V$ is a Taylor variety.
    \item The quotient of every compatible reflexive digraph in $\var V$ by its radical equivalence is cycle-free.
    \item For any compatible reflexive digraph in $\var V$, the strong equivalence coincides with the radical equivalence.
\end{enumerate}
\end{theorem}

In \cite{HaM}, Hagemann and Mitschke gave a characterization of  $n$-permutable varieties via compatible reflexive digraphs. They proved that a variety $\var V$ is $n$-permutable if and only if for every edge $a\to b$ of any compatible reflexive digraph $\rel G$ in $\var V$, there exists a directed  path with length at most $n-1$ from $b$ to $a$ in $\rel G$. We call a variety a {\em Hagemann-Mitschke variety} if it is $n$-permutable for some $n\geq 2$. In \cite{VW}, Valeriote and Willard proved that Hagemann-Mitschke varieties are the varieties whose full idempotent reduct does not interpret in the variety $\dl$ of distributive lattices. 
By the use of Theorem 4.4 in \cite{GMZ2}, we obtained a characterization of Hagemann-Mitschke varieties in terms of connectivity relations of compatible reflexive digraphs. The equivalence of items (1) and (2) in the following theorem is a part of this characterization. Moreover, it is not hard to see that items (2) and (3) are equivalent. One just has to use the facts that radically connected implies weakly connected, and that for a digraph, the weak equivalence is the smallest equivalence which yields a quotient that equals a disjoint union of loops. Note that, similarly as in Theorem \ref{th:taylor_equivalent_to_strong_equals_radical},  item (3) here also gives a direct link between equivalences related to two connectivity notions defined above.

\begin{theorem}[\cite{GMZ2}, cf. Corollary 4.5]\label{HoM} The following are equivalent for a variety $\var V$.
\begin{enumerate}
    \item  $\var V$ is a Hagemann-Mitschke variety. 
    \item The quotient of every compatible reflexive digraph in $\var V$ by its radical equivalence is a disjoint union of loops.
     \item For any compatible reflexive digraph in $\var V$, the  weak equivalence  coincides with the radical equivalence.
\end{enumerate}
\end{theorem} 

Hobby-McKenzie varieties were introduced in \cite{HM} as the varieties whose full idempotent reduct does not interpret in the variety $\slt$ of semilattices. Hobby-McKenzie varieties were characterized in various ways, see for example Theorem A.2 in \cite{KK}, and also Lemma 5.1 and Corollary 5.2 in \cite{BGMZ1}. At this point, we note that for the three types of varieties we introduced so far, we have the following implications (just by looking their characterizations via interpretation of their full idempotent reduct): $\var V$ is a Hagemann-Mitschke variety $\Rightarrow$ $\var V$ is a Hobby-McKenzie  variety $\Rightarrow$ $\var V$ is a Taylor variety.

In this article, we study the shape of compatible reflexive digraphs in Hobby-McKenzie varieties. Our goal is to present a similar characterization for  Hobby-McKenzie varieties as the ones for Taylor and  Hagemann-Mitschke varieties in the labeled theorems above. 


   

 In \cite{MZ}, the last two authors proved Theorem 2.9 which claims that  in a variety $\var V$ that has Gumm terms, every compatible finite strongly connected reflexive digraph is extremely connected. We note that existence of Gumm terms characterizes congruence modularity for a variety, and implies that the variety is Hobby-McKenzie.  We long have been conjectured that a generalization of Theorem 2.9 for Hobby-McKenzie varieties also holds. Namely, if $\var V$ is a Hobby-McKenzie variety, then every compatible strongly connected reflexive digraph in $\var V$ is extremely connected.  

In this article, we shall verify this conjecture.  In fact, we shall prove the stronger statement that a variety $\var V$ is a Hobby-McKenzie variety if and only if for any compatible reflexive digraph in $\var V$, the strong equivalence  coincides with the extreme equivalence. From this new description of Hobby-McKenzie varieties, it will then follow as a corollary that a variety $\var V$ is a Hagemann-Mitschke variety if and only if for any compatible reflexive digraph in $\var V$, the weak and the extreme equivalences coincide.

The structure of this paper is as follows. In Section 2,  we present some definitions related to varieties and  digraphs, and at the end of Section 2, prove some preliminary results related to two particular digraphs and Hobby-McKenzie varieties. In Section 3, after proving some preparatory lemmas and corollaries, we prove our main result on Hobby-McKenzie varieties, see Theorem \ref{main}, and, as a corollary, we give a better characterization of Hagemann-Mitschke varieties, see Corollary \ref{cor}. In Section 4, we summarize our results achieved in this article. 

\section{Preliminaries} 
In this section, we recall some basic algebraic notions related to varieties, clones, and interpretation of varieties. We introduce some further concepts of connectivity for digraphs. At the end of the section, we define some notions related to compatible digraphs, recall some well known facts, and verify statements which prove to be useful for our investigations in Section 3.

A term $t$ of a variety $\var V$ is called {\em idempotent} if $\var V$ satisfies the identity $t(x,\dots,x)=x$. The {\em  full idempotent reduct of a variety} $\var V$ is a variety  whose signature is the set of idempotent terms of $\var V$, and whose identities are those satisfied by the idempotent terms for $\var V$. Throughout the text, we denote the full idempotent reduct of $\var V$ by $\var V_{\Id}$. 

An identity in the language of a variety is {\em linear} if it has at most one occurrence of a function symbol on each side of the identity. It is well known that both Taylor, Hobby-McKenzie, and Hagemann-Mitschke varieties are characterized by the existence of a single $n$-ary idempotent term $t$ for some $n$ such that the variety satisfies a certain finite set of linear identities for $t$, see Theorems 5.1, 5.2, and 5.3 in \cite{GT}, Lemma 9.5 in \cite{HM}, and Theorem 4.2 in \cite{KKVW}. The corresponding term $t$ is respectively called  {\em a Taylor, a Hobby-McKenzie, and a Hagemann-Mitschke term}.  Now, it should be clear that a variety is respectively Taylor, Hobby-McKenzie, and Hagemann-Mitschke if and only if so is its idempotent reduct.

For later use, we define Hobby-McKenzie terms. An $n$-ary idempotent term $t$ of a variety $\var V$ is called a {\em Hobby-McKenzie term} if  for each $i\in\{1,\dots, n\}$, $\var V$ satisfies a linear identity in two variables $x$ and $y$ as of
$$t(x,\dots,x,x_{i+1},\dots, x_n)=t(y_1,\dots, y_{i-1},y,y_{i+1},\dots,y_n)$$ where $x_j,y_k\in\{x,y\}$ for all $i+1\leq j\leq n,$ and for all $\ k\neq i$ with $1\leq k\leq n$. In other words, $t$ is idempotent and obeys a set of linear identities that fails to hold for any $n$-ary term of $\mathcal{S}\mathcal{L}$.

Let $A$ be a set. A {\em clone on} $A$ is a set  of finitary operations on $A$ that contains all projection operations and is closed under composition. A {\em clone homomorphism} from a clone $\mathscr C$ to a clone $\mathscr D$ is a map from $\mathscr C$ to $\mathscr D$ that maps each projection of $\mathscr C$ to the corresponding projection of $\mathscr D$ and commutes with composition. Let $\alg A$ be an algebra.  The {\em clone of} $\alg A$ is the clone of finitary term operations of $\alg A$, and is denoted by $\clo(\alg A)$. 

 We define the notion of interpretability for varieties. Our notion of interpretability agrees with the one introduced in \cite{GT} by Garcia and Taylor. 
 A variety is called {\em trivial} if it consists of one-element algebras. The {\em clone of a trivial variety} is defined to be the clone of the one-element algebra. The {\em clone of a non-trivial variety} $\var V$ is the clone of the free algebra with countably infinite free generating set. 
 We denote this clone by $\clo(\var V)$. We say that {\em a variety $\var V$ interprets in a variety $\var W$} if there is a clone homomorphism from  $\clo(\var V)$ to $\clo(\var W)$. In the definition of interpretability, we may use the clone of any generating algebra of the varieties $\var V$ or $\var W$, since this clone is isomorphic to the clone of the respective variety.  We use this fact without any further note in the later proofs. We remark that interpretability of varieties can be expressed in terms of satisfaction for sets of identities. Thus the following is an equivalent description of interpretability: $\var V$ interprets in $\var W$ if and only if there is an arity preserving map $\alpha$ that sends the function symbols of $\var V$ to terms of $\var W$ such that for any set $\Sigma$ of identities of $\var V$, by replacing the function symbols  by their images under $\alpha$ in the identities of $\Sigma$, the resulting set of identities is satisfied by $\var W$. We call two varieties {\em equi-interpretable} if each of them interprets in the other. 

 For every digraph, we define four equivalences, each of which related to a certain kind of connectivity notion. These equivalences are proved to be useful when studying compatible digraphs in the three classes of varieties mentioned in the Introduction. Let $\rel G$ be a reflexive digraph.  Let $a$ and $b$ two vertices in $G$. An {$(a,b)$-path} is a sequence of vertices $a=a_0, a_1,\dots,a_n=b$ in $G$ such that $a_i\to a_{i+1}$ or  $a_{i+1}\to a_i$ for all $0\leq i<n$. A {\em directed $(a,b)$-path} is a sequence of vertices $a=a_0, a_1,\dots,a_n=b$ in $G$ such that $a_i\to a_{i+1}$  for all $0\leq i<n$. A {\em symmetric $(a,b)$-path} is a sequence of vertices $a=a_0, a_1,\dots,a_n=b$ in $G$ such that $a_i\to a_{i+1}$ and  $a_{i+1}\to a_i$ for all $0\leq i<n$. The {\em weak  equivalence} of $\rel G$ consists of all pairs $(a,b)\in G^2$ such that there is $(a,b)$-path in $\rel G$. The {\em strong equivalence} of $\rel G$ consists of  all pairs $(a,b)\in G^2$ such that there are directed $(a,b)$- and $(b,a)$-paths in $\rel G$. The {\em extreme equivalence} of $\rel G$ consists of  all pairs $(a,b)\in G^2$ such that there is a symmetric $(a,b)$-path in $\rel G$. The blocks are called the {\em weak}, the {\em strong}, and the {\em extreme components} of the respective equivalences. Now it is clear that a digraph $\rel G$ is weakly connected, strongly connected, and extremely connected if and only if the corresponding equivalence relation of $\rel G$ has a single block.

Let $\alpha$ be an equivalence on $\rel G$. Then $\rel G/\alpha$ denotes the digraph whose vertex set is $G/\alpha$ and edge set is defined by: $A\to B$ if and only if there is an edge $a\to b$ in $\rel G$ such that $a\in A$ and $b\in B$.
 We need one more notion of connectivity for digraphs. For a reflexive digraph $\rel G$, we define the {\em radical equivalence of}  $\rel G$ as follows. Let $\rel \nu_0$ denote the extreme equivalence of $\rel G$. Then  let $\nu'_1$ be the extreme equivalence of $\rel G/{\nu_0}$, and $\nu_1=\{(a,b)\colon a/\nu_0\ \nu'_1\ b/\nu_0\}.$ Clearly, $\nu_1$ is an equivalence on $G$, and $\rel G/\nu_1\cong (\rel G/\nu_0)/\nu'_1$ . By using $\nu_i$, we define $\nu_{i+1}$ similarly for each $i\geq 0$. Then clearly, $\nu_0\subseteq\dots\subseteq\nu_i\subseteq\dots$ are all equivalences on $ G$. So $\nu=\bigcup_{i=0}^\infty\nu_i$ is also an equivalence on  $ G$. We call $\nu$ the {\em radical equivalence} of $\rel G$. The blocks of $\nu$ are called {\em radical components}. Moreover, we say that $\rel G$ is {\em radically connected} if the radical equivalence of $\rel G$ has a single block. 
 
 \begin{proposition} Let $\rel G$ be a reflexive digraph. Then its radical equivalence $\nu$ is the smallest equivalence $\mu$ of $G$ such that $\rel G/\mu$ is antisymmetric.
\end{proposition}
 \begin{proof}
 In \cite{GMZ2} we proved that the digraph $\rel G/\nu$ is antisymmetric. Here we argue that $\nu$ is the smallest equivalence $\mu$ of $G$ such that $\rel G/\mu$ is antisymmetric.

 So  let $\mu$ be an equivalence of $G$ such that $\rel G/\mu$ is antisymmetric. We prove by an induction on $i$ that $\nu_i\subseteq\mu$ for all $i$, and hence $\nu\subseteq \mu$. If  $\nu_0\not\subseteq \mu$, then there is an $(a,b)\in \nu_0\setminus\mu.$
 Therefore, there exists a symmetric path $a=a_0\leftrightarrow a_1\leftrightarrow\dots \leftrightarrow a_n=b$ in $\mu$ and  a $0\leq j<n$ such that $(a_j,a_{j+1})\not\in \mu$, but this would imply that $a_j/\mu\neq a_{j+1}/\mu$ and $a_j/\mu\leftrightarrow a_{j+1}/\mu$, so $\rel G/\mu$ would not be antisymmetric. 
 So $\nu_0\subseteq \mu$.  Suppose that $i\geq 1$, then by the induction hypothesis $\nu_{i-1}\subseteq \mu$. If  $\nu_i\not\subseteq \mu$, then there is an $(a,b)\in \nu_i\setminus\mu.$ Therefore, there exist blocks $B_0,\dots, B_n$ of $\nu_{i-1}$, elements $a_k, b_k,a_k', b_k'\in B_k$ where $0\leq k\leq n$, and a $0\leq j<n$ such that $a\in B_0$, $b\in B_n$,  $a_k\to b_{k+1} $, $a_k'\leftarrow b_{k+1}'$ for all $0\leq k<n,$ and $(a_j,a_{j+1})\not\in \mu$. Since each of the $B_k$ are included in a $\mu$-block, this would imply again that $a_j/\mu\neq a_{j+1}/\mu$ and $a_j/\mu\leftrightarrow a_{j+1}/\mu$, so $\rel G/\mu$ would not be antisymmetric. Thus $ \nu_i\subseteq \mu$.
\end{proof}

Notice that by the preceding proposition, the radical equivalence is a subrelation of the strong equivalence of $\rel G$. Indeed, the quotient of $\rel G$ by the the strong equivalence is obviously a cycle-free, hence antisymmetric reflexive digraph. 

So it should now be clear that for any reflexive digraph $\rel G$, the weak, the strong, the radical, and the extreme equivalences comprise a decreasing sequence with respect to inclusion, as we have mentioned this fact in the Introduction. We also note that for every compatible digraph $\rel G$ of an algebra $\alg G$, the four equivalences defined by the different types of connectivity notions above are congruences of $\alg G$.

 For a digraph $\rel G$, the {\em $n$-th power $\rel G^n$} of $\rel G$ is the digraph with vertex set $G^n$ where the edges are defined by 
$(a_1,\dots,a_n)\to (b_1,\dots,b_n)$  if and only if for all $1\leq i\leq n$, $a_i\to b_i$ in $\rel G$. An $n$-ary operation $f$ of $G$ is called a polymorphism of $\rel G$ if $f$ preserves the edge relation of $\rel G$, that is, for all $(a_1,\dots,a_n)\to (b_1,\dots,b_n)$ in $\rel G^n$, $f(a_1,\dots,a_n)\to f(b_1,\dots,b_n)$ in $\rel G$. The polymorphisms of a digraph $\rel G$ form a clone on the set $G$. We denote this clone by $\pol(\rel G)$. The idempotent operations in $\pol(\rel G)$ also comprise a clone on $G$. This latter clone is denoted by $\pol_{\Id}(\rel G)$. We remark, and later use this fact without any further notice, that $\rel G$ is a compatible digraph in $\var V$ if and only if there is a clone homomorphism from the clone of $\var V$ to $\pol(\rel G)$.

 We define two particular small digraphs which appear frequently later in the proofs. Let $\rel D$ be the reflexive digraph on the 3-element set $\{0,1,2\}$ with non-loop edges $0~\leftrightarrow~1$, $1\to 2$, and $2\to 0$.   Let $\rel K$ be the reflexive digraph given with non-loop edges $0\leftrightarrow 1$, $1\to 2$, $2\leftrightarrow 3$, and $3\to 0$ on the 4-element set $\{0,1,2,3\}$. We depicted these two digraphs 
 in Figure \ref{DN}.

 \begin{figure}[H] 
\centering
\includegraphics[scale=1]{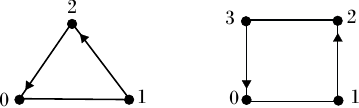}
\caption{The digraphs $\rel D$ and $\rel K$ (loop edges and arrows on double edges are not displayed).}
\label{DN}
\end{figure}

 In the following proposition, we establish some properties of $\rel D$ and $\rel K$.
 \begin{proposition}\label{DK} Let $\alg D$ and $\alg K$ denote the algebras defined respectively on $D$ and $K$ whose basic operations equal the idempotent polymorphisms of $\rel D$ and $\rel K$. Let $\var D$ and $\var K$ be the varieties respectively generated by $\alg D$ and $\alg K$. Then the following hold.
 \begin{enumerate}
 \item $\slt$ and $\var D$ are equi-interpretable. 
 \item $\set$ and $\var K$ are equi-interpretable.
\end{enumerate} 
 \end{proposition}
 \begin{proof}
 First we prove item (1). Observe that the meet-semilattice operation of the chain $0<1<2$ is a polymorphism of $\rel D$. So $\slt$ interprets in $\var D$.

 For the other direction, observe first that $\{0,2\}$ is closed under the operations of $\alg D$. Indeed, for any $c_1,\dots,c_n\in \{0,2\}$ and $n$-ary idempotent polymorphism $t$ of $\rel D$, $2=t(2,\dots,2)\to t(c_1,\dots,c_n)$. Hence $t(c_1,\dots,c_n)\in \{0,2\}$. Let $\alg D_0$ denote the subalgebra of $\alg D$ with underlying set $\{0,2\}$. Let $\var D_0$ be the variety generated by $\alg D_0$. Clearly, $\var D$ interprets in $\var D_0$. If we prove that 
 $\var D_0$ interprets in $\slt$, we are done.

For the proof of this, it suffices to verify that any term operation of $\alg D_0$ is a meet-semilattice operation of the chain $0<2$. Note that every term operation of $\alg D_0$ is of the form $t|_{\{0,2\}}$ where $t\in\pol_{\Id}(\rel D)$. Our initial goal is to describe the polymorphisms of $\pol_{\Id}(\rel D)$. So let $t$ be an $n$-ary idempotent polymorphism of $\rel D$. 
 

 We call $I\subseteq\{1,\dots,n\}$ {\em a major subset for $t$} if there exist $c_1,\dots,c_n\in D$ such that  \[t(c_1,\dots,c_n)=2\text{ and }
 I=\{i\mid c_i=2\}.
 \]
 By idempotency, $\{1,\dots,n\}$ is itself a major subset. We prove two claims on the major subsets.

 \begin{cla1}
 \label{c:2_iff_2_on_major} Let $d_1,\dots,d_n\in D$ and $I=\{i\colon d_i=2,\ 1\leq i\leq n\}$.
  Then   $t(d_1,\dots,d_n)=2$ if and only if $I$ is a major subset in $\{1,\dots,n\}$.
 \end{cla1}
 \begin{claimproof}
     The only if part is immediate from the definition of major subsets. For the if part, suppose that $I$ is a major subset, then there are $c_1,\dots,c_n\in D$ such that $$t(c_1,\dots,c_n)=2\text{ and } I=\{i\colon c_i=2, 1\leq i\leq n\}.$$ Notice that by the second equality,  $c_i\leftrightarrow d_i$ holds for all $1\leq i\leq n$. Therefore, as $t$ is a polymorphism, $t(c_1,\dots,c_n)\leftrightarrow t(d_1,\dots,d_n)$, which can only happen if $t(d_1,\dots,d_n)=2$. So the claim is proved.
 \end{claimproof}

 \begin{cla2}
 \label{c:majors_form_a_filter}
     The major subsets form a filter of the lattice $\alg {B}(n)$ of subsets of $\{1,\dots,n\}$.
 \end{cla2}
 \begin{claimproof}
     First we prove that if $I$ is a major subset, and $j\in\{1,\dots, n\}\setminus I$ then $I\cup\{j\}$ is also a major subset. Let $$c=(c_1,\dots,c_n)\in D^n \text{ such that } t(c)=2 \text{ and }
     I=\{i\mid c_i=2\}.$$ Then $c_j\in\{0,1\}$. Let $c'_j\in\{0,1\}$ so that $c_j\neq c'_j$, and let $$c'= (c_1,\dots,c'_j,\dots,c_n).$$
     Since $c_j\leftrightarrow c_j'$ and $t$ is a polymorphism, $2=t(c)\leftrightarrow t(c')$ and hence $t(c')=2$. Then by \[
     t(c_1,\dots,1,\dots,c_n)\to t(c_1,\dots,2,\dots,c_n)\to t(c_1,\dots,0,\dots,c_n),
     \] where $0, 1$ and $2$ are the $j$-th entries of the respective $n$-tuples, 
     \[
     2\to t(c_1,\dots,2,\dots,c_n)\to 2.
     \]
     Therefore, $t(c_1,\dots,2,\dots,c_n)=2$. So $I\cup\{j\}$ is a major subset. This yields that the major subsets form an upwardly closed subset of $\alg{B}(n)$.

     We still have to see that the major subsets are closed under intersection. Suppose that $I_1$ and $I_2$ are major subsets for $t$. Let \[
     e_i:=\begin{cases}
         2,\,\text{if }i\in I_1\cap I_2,\\
         1,\,\text{if }i\in I_1\backslash I_2,\\
         0,\,\text{otherwise}
     \end{cases}
     \] while \[
     e^+_i:=\begin{cases}
         2,\,\text{if }i\in I_1,\\
         0,\,\text{otherwise}
     \end{cases}
     \] and \[
     e^-_i:=\begin{cases}
         2,\,\text{if }i\in I_2,\\
         0,\,\text{otherwise.}
     \end{cases}
     \] Then $e^-_i\rightarrow e_i\rightarrow e^+_i$ for all $i$. By using Claim 1 and the fact that $t$ is a polymorphism, \[
     2=t(e^-_1,\dots,e^-_n)\rightarrow t(e_1,\dots,e_n)\rightarrow t(e^+_1,\dots,e^+_n)=2.
     \] Hence $t(e_1,\dots,e_n)=2$, and so $I_1\cap I_2$ is indeed a major subset for $t$.
 \end{claimproof}

 By Claim 2, there is a smallest major subset $I_t$ in $\{1,\dots,n\}$ for $t$. Notice that by Claims 1 and 2, for arbitrary $a_1,\dots,a_n\in D$, $t(a_1,\dots,a_n)=2$ holds if and only if $a_i=2$ for all $i\in I_t$. In particular, since $t$ is idempotent, $I_t\neq \emptyset$. So  for arbitrary $a_1,\dots,a_n\in \{0,2\}$, $t_{\{0,2\}}(a_1,\dots,a_n)=2$ holds if and only if $a_i=2$ for all $i\in I_t$. This means precisely that
$$t_{\{0,2\}}(x_1\dots,x_n)= \bigwedge_{i\in I_t} x_i.$$ Thus any term operation of $\alg D_0$ is a meet-semilattice operation of the chain $0<2$, as we claimed.
 
For item (2), it suffices to show that $\var K$ interprets in $\set$, since $\set$ interprets in every variety. The main result of \cite{LLP} applied to $\rel K$ asserts that all surjective polymorphisms of $\rel K$ are essentially unary, so the idempotent polymorphisms of $\rel K$ coincide with the projections. Thus $\var K$ interprets in $\set$.
\end{proof}

We define a natural construction of compatible digraphs. Let $\rel P$ be a digraph. Let $\alg A$ be an algebra such that $P\subseteq A$ and $P$  is a generating set of $\alg A$. We call the digraph $\rel A$ the {\em $\rel P$-generated digraph of $\alg A$}  if  the vertex set of  $\rel A$ coincides with $ A$, and the edge set of $\rel A$ equals the subalgebra generated by the edge set of $\rel P$ in $\alg A^2$.
Note that if $\rel A$ is the $\rel P$-generated digraph of $\alg A$, then $\rel A$ is a compatible digraph of $\alg A$. Let $\var V$ be a variety. Let $\alg F\in \var V$ be the free algebra with free generating  set $P$.  Sometimes, we call the $\rel P$-generated digraph of $\alg F$ the {\em compatible digraph freely generated by $\rel P$ in $\mathcal{V}$}. 

We say that a digraph $\rel H$ is {\em retract} of a digraph $\rel G$ if there exist   homomorphisms $\alpha\colon\rel G\to\rel H$ and $\beta\colon\rel H\to\rel G$  such that  $\alpha\beta=\id_H$. Then we call $\alpha$ a {\em retraction} and $\beta$ a  {\em coretraction}. We require the following lemma on freely generated digraphs.

\begin{lemma}[\cite{BGMZ1}, cf. Lemma 3.3] \label{freedigraph}
Let $\var V$ be a non-trivial variety. Let
$\rel S$ be a reflexive weakly connected digraph with vertex set $S=\{s_1,\dots,s_n\}$ such that  there exists a clone homomorphism from the clone of $\var V_{\Id}$ to $\pol(\rel S)$. Let $\alg F\in\var V$ denote the free algebra with free generating set $S$, and $\rel F$  the digraph freely generated by $\rel S$ in $\var V$. Let $U$ denote the set of unary term operations of $\alg F$, and for any $u\in U$, let $\rel F_u$ be the subdigraph of $\rel F$, induced by the subset
$$\{t(s_1,\dots,s_n)\colon t \text{ is an $n$-ary term where }\forall s\in S,\ t(s,\dots, s)=u(s)\}.$$ Then the following hold.
 
\begin{enumerate}
\item  The weak components of $\rel F$ coincide with the subdigraphs $\rel F_u$, $u\in U$.
\item  For the identity operation $\id$ of $F$,  $\rel F_{\id}$ is the weak component induced by the  idempotent term operations in  $\rel F$, and $\rel S$ is a retract of $\rel F_{\id}$ with the coretraction $\beta \colon S\to  F_{\id},\ a\mapsto a.$
\end{enumerate}
\end{lemma} 

The following corollary makes connection between the Hobby-McKenzie varieties and the digraph $\rel D$.
\begin{corollary}\label{Dretract}
    Let $\mathcal{V}$ be a variety, $\rel F$  the compatible digraph freely generated by $\rel D$ in $\mathcal{V}$, and $\rel F_{\id}$ the weak component that includes $D$ in $\rel F$.  Then the following are equivalent.
  \begin{enumerate}
      \item $\mathcal{V}$ is a Hobby-McKenzie variety.
      \item $\rel D$ is not a retract of $\rel F_{\id}$.
  \end{enumerate}
\end{corollary}
\begin{proof}
First we prove $(1) \Rightarrow (2)$.  If (1) holds, $\rel F$ has a Hobby-McKenzie polymorphism. This being an idempotent polymorphism restricts to $\rel F_{\id}$, so $\rel F_{\id}$ also has a Hobby-McKenzie polymorphism. Since linear identities are preserved under retract, 
if $\rel D$ would be a retract of $\rel F_{\id}$, then $\rel D$ would also have a Hobby-McKenzie polymorphism, so  $\var D$  would be a Hobby-McKenzie variety. On the other hand, $\var D$ is an idempotent variety, so $\var D$ and $\var D_{\Id}$ are equi-interpretable. So by item (1) in Proposition \ref{DK}, $\var D_{\Id}$ interprets in $\slt$. By using the characterization of Hobby-McKenzie varieties mentioned in the Introduction, this means that $\var D$  is in fact not a Hobby-McKenzie variety  This contradiction finishes the proof of  $(1) \Rightarrow (2)$.

For $(2) \Rightarrow (1)$, let us suppose that $\var V$ is not a Hobby-McKenzie variety. Then,  by using the characterization from the Introduction, $\mathcal{V_{\Id}}$ interprets in $\slt$. Hence, by the use of item (1) in Proposition \ref{DK} again, $\mathcal{V_{\Id}}$ interprets in $\var D$. So there is a clone homomorphism from the clone $\mathcal{V_{\Id}}$ to $\pol (\rel D)$. Then by the preceding lemma, $\rel D$ is a retract of $\rel F_{\id}$.
\end{proof}

\section{Main results} 

In this section, we give several new characterizations of Hobby-McKenzie varieties. As a corollary, we obtain a characterization of Hagemann-Mitschke varieties which yields a better insight into the structure of these varieties than the characterizations already known. We start with proving some lemmas that give us information on the shape of certain kinds of compatible digraphs in Hobby-McKenzie varieties. 

First we establish a property of the digraph freely generated by $\rel D$ in an idempotent Hobby-McKenzie variety.

\begin{lemma}\label{l:d_is_a_retract}
Let $\mathcal{V}$ be an idempotent Hobby-McKenzie variety, and $\rel F$  the compatible digraph freely generated by $\rel D$ in $\mathcal{V}$. Then $\rel F$ is extremely connected.    
\end{lemma}
\begin{proof}  Let $\mathcal{V}$ be an idempotent variety, and assume that $\rel F$ is not extremely connected. We shall prove that $\rel D$ is a retract of $\rel F$. By Corollary \ref{Dretract}, this means that $\var V$ is not a Hobby-McKenzie variety. 

We denote  the extreme equivalence on $\rel F$ by $\sim$. Let $\alg F$ be the free algebra with free generating set $\{0,1,2\}$ in $\var V$. As we have mentioned, $\sim$ is a congruence of $\alg F$, so the term operations of $\alg F$ preserve $\sim$. Let $\rel T$ denote the extreme component of 2 in $\rel F$. The proof  that $\rel D$ is a retract of $\rel F$ goes through a series of claims.
\begin{cla1}
\label{c:0_not_sim_2}
    $0,1\not\in T$.
\end{cla1}
\begin{claimproof}
    Suppose that $0\sim 2$. Then the elements of $D$ are in the same $\sim$-class. Each element of $F$ is of the form $t(0,1,2)$ for some ternary term $t$ of $\mathcal{V}$. As $\sim$ is a congruence of $\alg F$, $t(0,1,2)\sim t(0,0,0)=0$ and hence all elements of $F$ are in the $\sim$-class of $0$, a contradiction. Thus $0\not\in T$. Since $0\sim 1$, we also have $1\not\in T$.
\end{claimproof}

\begin{cla2}
\label{c:polynomial_in_top_component}
Let $t$ be a 4-ary term of $\mathcal{V}$, and let $p(x)=t(0,1,2,x)$. If $p(0)\in T$, then $p(2)\in T$.
\end{cla2}
\begin{claimproof}
    By using that $t$ is idempotent and acts on $F$ as a polymorphism of $\rel F$
    \begin{multline*} p(0)=t(t(0,1,2,0),t(0,1,2,0),t(0,1,2,0),t(0,1,2,0))\rightarrow \\
        t(0,0,t(1,1,2,1),t(1,1,2,1))\rightarrow t(0,1,2,0)=p(0),
    \end{multline*} and \[
    t(0,0,t(1,1,2,1),t(1,1,2,1))\leftrightarrow t(0,1,t(0,1,2,0),t(0,1,2,0)).
    \] So $p(0)\sim t(0,1,p(0),(p(0))$. By using that $\sim$ is a congruence and $p(0)\sim 2$, we obtain \[
    p(0)\sim t(0,1,p(0),p(0))\sim t(0,1,2,2)=p(2),
    \] hence $p(2)\in T$.
\end{claimproof}

\begin{cla3}
\label{c:in-edge_to_top_component}
    \noindent \begin{enumerate}
        \item Let $u\in F$, $v\in T$ and $u\rightarrow v$. Then there is a 4-ary term $t$ of $\mathcal{V}$ such that $u=p(1)$ and $p(2)\in T$ hold for the unary polynomial $p(x)=t(0,1,2,x)$.
        \item Let $u\in F$, $v\in T$, and $u\leftarrow v$. Then there is a 4-ary term $t$ of $\mathcal{V}$ such that $u=p(0)$ and $p(2)\in T$ hold for the unary polynomial $p(x)=t(0,1,2,x)$.
    \end{enumerate}
\end{cla3}
\begin{claimproof}
    We only prove the first statement, as the second is an obvious dual of the first. As $u\to v$ in $\rel F$, by the definition of $\rel F$ there is a 7-ary term $s$ of $\mathcal{V}$ such that \[
    \begin{array}{c}
         u = s(0, 1 ,2, 0 ,1 ,1 ,2), \\
         v = s(0,1,2,1, 0, 2, 0).
    \end{array}
    \]
    Let $t(x,y,z,u):=s(x,y,z,x,y,u,z)$ and $v':=t(0,1,2,2)$. Then \[
    \begin{array}{c}
         p(1)=t(0,1,2,1)= s(0, 1 ,2, 0 ,1 ,1 ,2)=u, \\
         p(2)=t(0,1,2,2)= s(0, 1 ,2, 0 ,1 ,2 ,2)=v'.
    \end{array}
    \] 
    Thus $u\to v'$ in $\rel F$. 
    
    We argue that $v'\in T$. Let  $t'(x,y,z,u):=s(x,y,z,y,x,z,u)$ and $p'(x)=t'(0,1,2,x)$. So \[
    p'(0)=t'(0,1,2,0)=s(0,1,2,1,0,2,0) = v \sim 2.
    \] 
    Hence by applying Claim 2  for $t'$ and $p'$, we obtain that 
    \[
    v'=s(0,1,2,0,1,2,2) \sim s(0,1,2,1,0,2,2)=t'(0,1,2,2)=p'(2)\sim 2,
    \] which concludes the proof of Claim 3.
\end{claimproof}

\begin{cla4}
\label{no_both-way_neighbor_of_2sim}
    Let $w\in F$ and  $w^+,w^-\in T$ such that $w^-\rightarrow w\rightarrow w^+$. Then $w\in T$.
\end{cla4}
\begin{claimproof}
    By Claim 3, there are 4-ary terms $t$ and $s$ of $\mathcal V$ such that $t(0,1,2,1)=s(0,1,2,0)=w$ and $t(0,1,2,2), s(0,1,2,2)\in T$. Now
    \begin{multline*}
    w=t(0,1,2,1)=s(t(0,1,2,1),t(0,1,2,1),t(0,1,2,1),t(0,1,2,1))\rightarrow \\ s(0,0,t(1,1,2,2),t(1,1,2,2))\rightarrow \\ s(0,1,t(2,2,2,2),t(0,0,0,0))=s(0,1,2,0)=w,
    \end{multline*}
    so 
    \begin{multline*}
    w\sim s(0,0,t(1,1,2,2),t(1,1,2,2))\sim \\s(0,1,t(0,1,2,2),t(0,1,2,2))\sim s(0,1,2,2)\in T,
    \end{multline*}
    which finishes the proof of Claim 4.
\end{claimproof}

 Let $R$ denote  the set of vertices in $F\setminus T$ from which there is an edge in $\rel F$ to some vertex of $T$, and let $P=F\setminus (T\cup R)$. Now we  define a map $\alpha$ from $\rel F$ to $\rel D$ by
\[ \alpha(x) = \begin{cases} 2 & \mbox{if } x\in T, \\ 1 & \mbox{if } x\in R, \\ 0 &  \mbox{if } x\in P. \end{cases} \] 
Then, by Claim 1, $1\in R$ and by Claims 1 and 4, $0\in P$. By Claim 4, $\rel F$ has no edges from $T$ to $R$, and by the definition of $\alpha$,  $\rel F$ has no edges from $P$ to $T$. Hence $\alpha$ is a retraction from $\rel F$ onto $\rel D$. Thus $\rel D$ is a retract of $\rel F$.
\end{proof}

\begin{corollary}\label{cor:hm_means_d_connected_is_extremely_connected}
Let $\var V$ be a Hobby-McKenzie variety.   Let $\alg A\in \var V$  be any algebra with generating set $D\subseteq A$. Let $\rel A$ be the $\rel D$-generated digraph of $\alg A$. Then all weak components of $\rel A$ are extremely connected. 
\end{corollary}
\begin{proof}  Let $\alg F\in \var V$ be the free algebra with free generating set $D$, and $\rel F$ the compatible digraph freely generated by $\rel D$ in $\var V$. First we prove that the weak components of $\rel F$ are extremely connected.
Notice that by Lemma \ref{freedigraph}, the compatible digraph freely generated by $\rel D$ in $\var V_{\Id}$ coincides with the weak component $\rel F_{\id}$ induced by the idempotent ternary terms in $\rel F$. Then by the preceding lemma, $\rel F_{\id}$ is extremely connected. Now we prove that any weak component of $\rel F$ is extremely connected.  By Lemma \ref{freedigraph}, we know that any weak component of $\rel F$ is of the form $\rel F_u$ for some
unary term operation $u$ of $\alg F$ and consists of all elements $t(0,1,2)\in F$ where  $t$ is a ternary term operation of $\alg F$ and for all  $0\leq i\leq2$,  $t(i,i,i)=u(i)$. Let $t(0,1,2)\in F_u$ for such $t$. As $\rel F$ is a compatible digraph of $\alg F$ and $t$ is a polymorphism of $\rel F$, $t|_{F_{\id}^3}$ is a homomorphism from $\rel F_{\id}^3$ to $\rel F_u$.  Since $\rel F_{\id}$ is extremely connected, so are $\rel F_{\id}^3$  and $t(\rel F_{\id}^3)$. Therefore, for any ternary term operation $t$ of $\alg F$ where $t(i,i,i)=u(i)$ for all  $0\leq i\leq2$, $t(0,1,2)$ and $t(0,0,0)$ are connected by a symmetric path in $\rel F_u$. Thus $\rel F_u$ is extremely connected.

Now we prove that the weak components of $\rel A$ are extremely connected. Let $\alpha$ be the homomorphism from $\alg F$ onto $\alg A$ such that $\alpha(d)=d$ for all $d\in D$. 

First we prove that $\alpha$ is also a homomorphism from $\rel F$  onto $\rel A$. Any edge of $\rel F$ is of the form $t_{\alg F}(0,1,2,1,0,1,2)\to t_{\alg F}(0,1,2,0,1,2,0)$ where $t$ is a 7-ary term of $\var V$. 
Since $\alpha$ is a homomorphism and $\alpha(i)=i$ for all $i\in D$,  
\begin{align*}\alpha(t_{\alg F}(0,1,2,1,0,1,2))&=t_{\alg A}(0,1,2,1,0,1,2)\text{ and }\\
\alpha(t_{\alg F}(0,1,2,0,1,2,0))&=t_{\alg A}(0,1,2,0,1,2,0). 
\end{align*}
In $\rel A$, we have that $t_{\alg A}(0,1,2,1,0,1,2)\to t_{\alg A}(0,1,2,1,0,1,2)$. So $\alpha$ preserves the edges of $\rel F$.

Now we give a proof that for every edge of the digraph $\rel A$ is of the form $\alpha(f)\to \alpha(g)$ where $f\to g$ in $\rel F$.
Let $a\to b$ be an arbitrary edge in $\rel A$, then there is 7-ary term $t$ of $\var V$ such that $a=t_{\alg A}(0,1,2,1,0,1,2)$ and $b=t_{\alg A}(0,1,2,0,1,2,0)$ in $\alg A$. By using that $\alpha(i)=i$ for all $i\in D$, we also have
 \begin{align*}
 a&=t_{\alg A}(\alpha(0),\alpha(1),\alpha(2),\alpha(1),\alpha(0),\alpha(1),\alpha(2))\text{ and }\\b&=t_{\alg A}(\alpha(0),\alpha(1),\alpha(2),\alpha(0),\alpha(1),\alpha(2),\alpha(0)).   
 \end{align*}
 Since $\alpha$ is a homomorphism,
 $$a=\alpha(t_{\alg F}(0,1,2,1,0,1,2))\text{ and }b=\alpha(t_{\alg F}(0,1,2,0,1,2,0)).$$ Now, via the preceding two equalities, the two vertices $$f=t_{\alg F}(0,1,2,1,0,1,2)\text{ and } g=t_{\alg F}(0,1,2,0,1,2,0)$$ 
  witness the fact that $a\to b$ is of the form $\alpha(f)\to \alpha(g)$ where $f\to g$ in $\rel F$.
 
 Suppose that there is an oriented path from  $a$ to $b$ in $\rel A$, i.e., there is sequence $a=a_0, \dots,a_n=b$ such that $a_i\to a_{i+1}$ or $a_{i+1}\to a_i$ in $\rel A$ for all $0\leq i<n$. Since any edge of $\rel A$  is the image of an edge of $\rel F$ under $\alpha$, for every $0\leq i<n$, there is a weak component of $\rel F$ whose image under $\alpha$ contains $a_i$ and $a_{i+1}$. Since the weak components of $\rel F$ are extremely connected, so are their homomorphic images. Therefore, for each $0\leq i<n$, there is a symmetric path connecting $a_i$ and $a_{i+1}$ in $\rel A$. So $a$ and $b$ are connected by a symmetric path in $\rel A$. Thus the weak components of $\rel A$ are extremely connected.
\end{proof}

    Let $\rel G$ and $\rel H$ be digraphs. A sequence $g=g_0,g_1,\dots,g_n=g'$ such that for all $0\leq i<n$, there is a homomorphism $\varphi_i\colon \rel H \to \rel G $  whose range contains $g_i$ and $g_{i+1}$ is called an {\em $\rel H$-path from $g$ to $g'$ in $\rel G$}. The equivalence that contains the pairs $(g,g')$ where $g$ and $g'$ are connected by an $\rel H$-path is called the {\em $\rel H$-equivalence of $\rel G$}. The blocks of the $\rel H$-equivalence of $\rel G$ are called the {\em $\rel H$-components of $\rel G$}. We say that $\rel G$ is {\em $\rel H$-connected} if the $\rel H$-equivalence of $\rel G$ has a single block.   
   
Let $\rel N= (\{0,1\};\ \{(0,0),(1,1),(0,1),(1,0)\})$. We note that for any reflexive digraph, the extreme equivalence coincides with the $\rel N$-equivalence. Moreover, for any reflexive digraph, the extreme equivalence is a subrelation of $\rel H$-equivalence if  $\rel H$ is an at least 2-element digraph.

 In \cite{O}, Ol\v s\'ak proved that every Taylor variety $\var V$ has an idempotent 6-ary term $o$ (we call $o$ an {\em Ol\v s\'ak term}) such that $\var V$ satisfies the following identities $$o(x,x,x,y,y,y)=o(x,y,y,x,x,y)=o(y,x,y,x,y,x).$$

\begin{lemma}
\label{l:Olsak_means_k_connected_is_d_connected}
    In a Taylor variety $\var V$, for any compatible reflexive digraph $\rel G$ in $\var V$, the  $\rel K$-equivalence of $\rel G$ coincides with the $\rel D$-equivalence of $\rel G$.
\end{lemma}
\begin{proof} 
Since there is a homomorphism from $\rel K$ to $\rel D$, the $\rel D$-equivalence is a subrelation of the $\rel K$-equivalence for any digraph. Let $\var V$ be a Taylor variety, and $\rel G$ a compatible reflexive digraph in $\var V$. Let $o$ be an Ol\v s\'ak polymorphism of $\rel G$. Let $\rel C$ be a $\rel K$-component of $\rel G$. To complete the proof of the lemma, we shall prove that $\rel C$ is included in a $\rel D$-component of $\rel G$.

Let $\varphi$ be a homomorphism from $\rel K$ to  $\rel C$ such that $$\varphi(0)=a,\ \varphi(2)=b,\ \varphi(3)=c \text{ and }\varphi(4)=d.$$ 
So for $a,b,c,d\in C$ have that $a\leftrightarrow b\rightarrow c\leftrightarrow d\rightarrow a$. First we prove that $a$ and $c$ are in the same $\rel D$-component of $\rel G$. Since $o$ is a polymorphism of $\rel G$, 
    \begin{align*}
    &a\leftrightarrow o(a,a,a,b,b,b)\rightarrow o(a,a,a,c,c,a)\leftrightarrow\\
    &o(a,a,a,d,d,a)\rightarrow o(a,b,b,a,a,b)=o(a,a,a,b,b,b).
    \end{align*}
    This implies that $a$, $o(a,a,a,b,b,b)$, and $o(a,a,a,d,d,a)$ are in the same $\rel D$-component. By
    \begin{align*}
    &o(a,a,a,d,d,a)\leftrightarrow o(a,a,a,c,c,a)\leftrightarrow o(b,a,b,d,c,a)\rightarrow\\ 
    &o(c,a,c,a,c,a)=o(a,a,a,c,c,c)\leftrightarrow o(b,a,a,d,c,d)\rightarrow o(b,a,b,d,c,a),
    \end{align*}
    $o(a,a,a,c,c,c)$ is also in this $\rel D$-component. By a similar argument, $c$ and $o(a,a,a,c,c,c)$ are in the same $\rel D$-component. Thus $a$ and $c$ are in the same $\rel D$-component of $\rel G$.

     Since $\rel C$ is $\rel K$-connected and the ranges of homomorphisms from $\rel K$ to $\rel C$ are included in the same $\rel D$-component, $\rel C$ is a subdigraph of a $\rel D$-component. So the $\rel K$-equivalence of $\rel G$ is included in the $\rel D$-equivalence of $\rel G$. 
\end{proof}

By Corollary \ref{cor:hm_means_d_connected_is_extremely_connected}, the preceding lemma yields the following.

\begin{corollary} \label{cor:hm_means_k_connected_is_extremely_connected}
In a Hobby-McKenzie variety $\var V$, for any compatible reflexive digraph, the $\rel K$-equivalence coincides with the extreme equivalence.
\end{corollary}
\begin{proof} Let $\alg G\in \var V$, and $\rel G$ a compatible reflexive digraph of $\alg G$. As it is obvious that the $\rel K$-equivalence of $\rel G$ includes the extreme equivalence of $\rel G$, it suffices to prove that every $\rel K$-component of $\rel G$ is extremely connected.

Let $\rel C$ be a $\rel K$-component of $\rel G$. Then by the preceding lemma  $\rel C$ is a $\rel D$-component of $\rel G$.
Let $g$ and $g'$ be two arbitrary vertices of $C$. By using that
$\rel C$ is $\rel D$-connected, there exist a sequence
$g=g_0,g_1,\dots,g_n=g'$ in $\rel C$ and for each $0\leq i<n$, a homomorphism $\varphi_i\colon \rel D \to \rel G $  whose range contains $g_i$ and $g_{i+1}$. Let $\rel D_i$ be the subdigraph induced by $\varphi_i(D)$ in $\rel C$. Let $\rel G_i$ be the $\rel D_i$-generated digraph of $\alg G_i$ where $\alg G_i$ is the subalgebra of $\alg G$ generated by $D_i$. Notice that by compatibility of $\rel G$,  $\rel G_i$ is subdigraph of $\rel G$. 

Now we argue that any two vertices of $\rel D_i$ are connected by a symmetric path in $\rel G$. If $\rel D_i$ is not isomorphic to $\rel D$, then, since there is onto homomorphism from $\rel D$ to $\rel D_i$,  $\rel D_i$ itself is extremely connected. On the other hand, if $\rel D_i$ is isomorphic to $\rel D$, then by Corollary \ref{cor:hm_means_d_connected_is_extremely_connected}, the weak component of $\rel G_i$ that includes $D_i$ is extremely connected. Thus in either ways, any two vertices of $\rel D_i$ are connected by a symmetric path in $\rel G$. 

This yields that $g$ and $g'$ are connected by a symmetric path in $\rel G$. Thus $\rel C$ is extremely connected.
\end{proof}

We use the preceding corollary to give a characterization of Hobby-McKenzie varieties. For any positive integer $n$, let $\rel C_n$ denote the reflexive directed $n$-cycle.

\begin{theorem} \label{main}
    For any variety $\var V$, the following are equivalent.
    \begin{enumerate}
        \item $\mathcal V$ is a Hobby-McKenzie variety.
        \item For any compatible reflexive digraph in $\var V$, the radical equivalence coincides with the extreme equivalence.
        \item For any compatible reflexive digraph in $\var V$, the strong equivalence coincides with the extreme equivalence.
        \item For any positive integer $n$, in the digraph freely generated by $\rel C_n$ in $\var V$, the weak component of $\rel C_n$ is extremely connected.
        \item In the digraph freely generated by $\rel C_3$ in $\var V$, the weak component of $\rel C_3$ is extremely connected.
        \item There exist 6-ary terms $s_i$  and $t_i$ , $1\leq i\leq n$, for $\var V$ such that $\var V$  satisfies the following identities
        
        \begin{align*}  
            t_1(x,x,y,y,z,z) &= x, \\
            t_i(x,x,y,y,z,z) &= s_i(x,y,y,z,z,x)\text{ for all } 1\leq i\leq n,\\
            s_i(x,x,y,y,z,z) &= t_i(x,y,y,z,z,x)\text{ for all } 1\leq i\leq n,\\
            t_{i}(x,x,y,y,z,z) &= t_{i-1}(x,y,y,z,z,x) \text{ for all } 1<i\leq n,\\
            y &= t_n(x,y,y,z,z,x).
        \end{align*}
    \end{enumerate}
\end{theorem}
\begin{proof}
Due to the fact that $\rel C_n$ is a strongly connected digraph,
the weak component of  $\rel C_n$ coincides with the strong component of  $\rel C_n$ in the compatible digraph freely generated by $\rel C_n$ in any variety $\var V$. Hence the implications $(3)\Rightarrow (4)\Rightarrow (5)$ are obvious. Notice that $(6)$ holds if and only if in the free digraph freely generated by the isomorphic copy $(\{x,y,z\};\{(x,x),(x,y),(y,y),(y,z),(z,z),(z,x)\})$ of $\rel C_3$ in $\var V$, there is a symmetric path of length $n$ from $x$ to $z$. Hence, $(5)\Rightarrow(6)$.

For  $(1)\Rightarrow(2)$, we assume that $\alg G\in\var V$, $\rel G$ is a compatible reflexive digraph of $\alg G$, and $\rel H$ is a radical but not an extreme component of $\rel G$. Hence, there are extreme components $\rel A$ and $\rel B$ in $\rel H$ such that in the quotient of $\rel G$ by its extreme congruence, we have $A\leftrightarrow B$. This means that there are vertices $a,a'\in A$ and $b,b'\in B$ such that $a\rightarrow b$ and $b'\rightarrow a'$. Let $d$ be the maximum of the lengths of a shortest symmetric $(a,a')$-path and  a shortest symmetric $(b,b')$-path.  We assume that $\rel G$, $\rel H$, $a,a',b$, and $b'$ are chosen so that $d$ is minimal.

Clearly, $d$ is not $0$, as this would imply that $a=a'$ and $b=b'$, so there would be a double edge between $a$ and $b$ in $\rel G$. If $d=1$, 
then the subdigraph of $\rel G$ induced by $\{a,a',b,b'\}$ is isomorphic to $\rel D$ or $\rel K$. So $(a,b)$ is in the $\rel K$-equivalence of $\rel G$. Then by Corollary \ref{cor:hm_means_k_connected_is_extremely_connected}, $a$ and $b$ are in the same extreme component of $\rel G$, a contradiction.

Hence, $d>1$. Then we define the reflexive relation $\rho$ on $G$ by \[
(x,y)\in\rho\,\Leftrightarrow\,\exists u: x\rightarrow u \leftrightarrow y.
\]

Clearly, $\rho$ is a subalgebra of $\alg G^2$. Hence $\rel G'=(G;\rho)$ is a compatible reflexive digraph in $\var V$. Notice that by the reflexivity of $\rel G$, the edges of $\rel H$ are contained in $\rho$. Let $\rel H'$ be the strong component of $\rel G'$ including $\rel H$.
Clearly, if $x\leftrightarrow y\leftrightarrow z$ in $\rel G$, then $(x,z)\in\rho$. Therefore, the maximum of the lengths of the shortest symmetric $(a,a')$- and $(b,b')$-paths in $\rel H'$ is smaller than $d$.

Now by the minimality of $d$, the vertices $a$ and $b$ are in the same extreme component of $\rel H'$. So there exists a symmetric $(a,b)$-path in $\rel H'$. Notice that if $(x,y),(y,x)\in\rho$, then $x$ and $y$ are in the range of a homomorphism from $\rel K$ to $\rel G$. Consequently, $(a,b)$ is in the  $\rel K$-equivalence, and hence by Corollary \ref{cor:hm_means_k_connected_is_extremely_connected} there exists a symmetric $(a,b)$-path in $\rel G$, a contradiction.

 Now we prove that $(2)\Rightarrow (3)$. If $\var V$ is a Taylor variety, then, by Theorem \ref{th:taylor_equivalent_to_strong_equals_radical}, all compatible strongly connected digraphs in $\var V$ are radically connected, and hence by (2), they are extremely connected. So it suffices to prove that $\var V$ is a Taylor variety. Let us suppose that $\var V$ is a not Taylor variety, then $\var V_{\Id}$ interprets in $\set$, hence in $\var D$ as well.
 Let $\rel F_0$ be the digraph freely generated by $\rel D$ in $\var V_{\Id}$. We note that $$F_0=\{t(0,1,2)\colon t \text{ is a ternary idempotent term of }\var V \}.$$
  Then, by Lemma 3.3 in \cite{BGMZ1}, $\rel F_0$ is a weak component of the digraph $\rel F$ freely generated by $\rel D$ in $\var V$, and $\rel D$ is a retract of $\rel F_0$.  Let $\nu$ be the extreme equivalence of $\rel F_0$. For any ternary idempotent term $t$ of $\var V$,  as $t$ acts on $F$ as a polymorphism of $\rel F$, $$1=t(1,1,1)\to t(0,1,2)\to t(0,0,0)=0.$$ Then, by $(0,1)\in \nu$, for every idempotent term $t$ of $\var V$, 
  $$t(0,0,0)/\nu\leftrightarrow t(0,1,2)/\nu\text{ in } \rel F_0/\nu.$$ 
  So the extreme equivalence of $\rel F_0/\nu$ is the full relation. Therefore, $\rel F_0$ is radically connected, and so $F_0$ is a block of the radical equivalence of $\rel F$. Hence by (2), $\rel F_0$ is extremely connected. So its retract $\rel D$ must also be extremely connected, a contradiction. Thus $\var V$ is a Taylor variety.

Finally, we prove that $(6)\Rightarrow(1)$.  Let suppose that item (6) holds, and $\var V$  is still not a Hobby-McKenzie variety. Then $\var V_{\Id}$ interprets in $\slt$. Clearly, all terms which occur in the identities of (6) are idempotent. Consequently, (6) holds in $\slt$. Then by the remark at the beginning of the proof  of the present theorem, the compatible digraph freely generated by $\rel C_3$ in $\slt$ must have a symmetric path from $x$ to $z$. This freely generated digraph has seven vertices, and is easy to construct, see Figure \ref{freely_generated_by_c3_in_semilattice}. Evidently, it has no symmetric path from  $x$ to $z$, a contradiction.
\end{proof}





\begin{figure}[H] 
\centering
\includegraphics[scale=1]{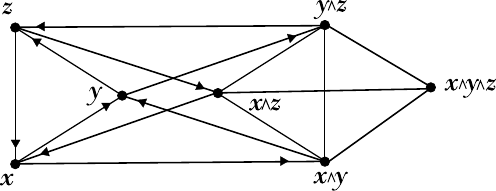}
\caption{The digraph freely generated by $\rel C_3$ in the variety of semilattices}
\label{freely_generated_by_c3_in_semilattice}
\end{figure}

The preceding theorem yields a new  characterization of Hagemann-Mitschke varieties.

\begin{corollary}\label{cor} For any variety $\var V$, the following are equivalent.
\begin{enumerate}
    \item  $\var V$ is a Hagemann-Mitschke variety. 
    \item For any compatible reflexive digraph in $\var V$, the weak equivalence coincides with the strong equivalence.
    \item For any compatible reflexive digraph  in $\var V$, the weak equivalence coincides with the radical equivalence.
    \item For any compatible reflexive digraph in $\var V$, the weak equivalence coincides with the extreme equivalence.
\end{enumerate}
\end{corollary} 
\begin{proof}
By the characterization of $n$-permutable varieties of Hagemann and Mitschke that we mentioned in the Introduction, $(1)\Rightarrow (2)$ is clear. Hagemann-Mitschke varieties are Hobby-McKenzie varieties, hence by
the equivalence of the first  two conditions in the preceding theorem, $(2) \Rightarrow (4)$ is also clear. Since the extreme equivalence is included in the radical equivalence, $(4)\Rightarrow (3)$.  Finally, by Theorem \ref{th:taylor_equivalent_to_strong_equals_radical}, we obtain  $(3)\Rightarrow (1)$.
\end{proof}
\section{Conclusion}
We summarize the main results we achieved in this article. We gave various characterizations of the Taylor, the Hobby-McKenzie, and the Hagemann-Mitschke varieties by the use of four types of connectivity notions for their compatible digraphs. It turned out that a variety is Hagemann-Mitschke if and only if for every compatible digraph in the variety, the weak equivalence coincides with  either the strong, the radical or the extreme equivalence. A variety is Hobby-McKenzie if and only if for every compatible digraph in the variety, the extreme equivalence coincides with  the strong or the radical equivalence. A variety is Taylor if and only if  for every compatible digraph in the variety, the strong equivalence coincides with  the radical equivalence. Note that there are only 6 ways to collapse two equivalences defined by the 4 connectivity notions we introduced, and our results obtained  in this article exhaust all the 6 cases.


\begin{thebibliography}{}
\bibitem{BGMZ0} Bodor, B, Gyenizse, G, Mar\'oti, M and Z\'adori, L; Taylor is prime, IJAC, 34/06, 857--879 (2024).
\bibitem{BGMZ1} Bodor, B, Gyenizse, G, Mar\'oti, M, and Z\'adori, L; The filter of interpretability types of Hobby-McKenzie varieties is prime, submitted https://www.math.u-szeged.hu/~zadori/publications/publ37.pdf (2024).
\bibitem{GT} Garcia O. C and Taylor, W; The lattice of interpretability types of varieties, Mem. Amer. Math. Soc. 50 (1984) v+125.
\bibitem {GMZ2} Gyenizse, G, Maróti, M, and Zádori, L; Reflexive digraphs in Taylor varieties, submitted https://www.math.u-szeged.hu/~zadori/publications/publ38.pdf (2025).
\bibitem{HaM} Hagemann, J and Mitschke, A; On n-permutable congruences, Algebra Universalis 3, 8--12 (1973)
\bibitem{HM} Hobby, D and McKenzie, R; The structure of finite algebras, Contemporary Mathematics 76, American Mathematical Society, Providence, RI, 1988.
\bibitem{KK} Kearnes, K. A and  Kiss, E. W; The shape of congruence lattices, Mem. Amer. Math. Soc. 222 (2013) viii+169.
\bibitem{KKVW} Kozik, M, Krokhin, A, Valeriote, M, and Willard, R; Characterizations of several  Maltsev conditions, Algebra Universalis 73, 205--224 (2015).
\bibitem{LLP} Larivi\` ere, I, Larose, B, and  Pazmi\~ no Pullas, D E; Surjective Polymorphisms of Directed Reflexive Cycles, Algebra Universalis 85/4, 1--28 (2024).
\bibitem{MZ} Mar\'oti, M and Z\'adori, L; Reflexive digraphs with near unanimity polymorphisms, Discrete Mathematics 12/15, 2316--2328 (2012). 
\bibitem{O} Ol\v s\'ak, M; The weakest nontrivial idempotent equations, Bul. Lond. Math. Soc. 49/6, 1028--1047 (2017).
\bibitem{T} Taylor, W; Varieties obeying homotopy laws, Canad. J. Math., 29/3, 498--527 (1977).
\bibitem{VW} Valeriote, M and Willard, R; Idempotent $n$-permutable varieties, Bulletin of the London Mathematical Society 46, 870--880, (2014).
\end{thebibliography}
\end{document}